\documentclass{article}
\usepackage{amsmath}
\usepackage{amsthm}
\usepackage{amssymb}
\usepackage{mathrsfs}
\usepackage{relsize}
\theoremstyle{definition}
\newtheorem{definition}{Definition}[section]
\theoremstyle{plain}
\newtheorem{theorem}[definition]{Theorem}
\newtheorem{corollary}[definition]{Corollary}
\newtheorem{lemma}[definition]{Lemma}
\usepackage{biblatex}
\title{Well-Orderedness of the Bashicu Matrix System}
\author{Rachel Hunter}

\begin{document}
\maketitle
\begin{abstract}
The Bashicu Matrix System is a recursive system of ordinal notations created by the user BashicuHyudora of the japanese Googology Wiki. In this paper, we prove that the Bashicu Matrix System is well-ordered.
\end{abstract}

\section{Introduction}
The Bashicu Matrix System ($BMS$) is a recursive system of ordinal notations with a large order type created by the user BashicuHyudora of the japanese Googology Wiki \cite{bashicu}. Originally, it was defined informally in pseudocode based on the programming language BASIC, and the following is the agreed-upon formalization:

\begin{definition}
\label{bmsdef}
An array is a sequence of equal-length sequences of natural numbers, i.e. an element of ${(\mathbb{N}^n)}^m$ for some $n,m\in\mathbb{N}$. For every array $A\in{(\mathbb{N}^n)}^m$, the columns of $A$ are its elements, and for each $n'<n$, the $n'$-th row of $A$ is the sequence of length $m$ such that for each $m'<m$, the $m'$-th element of the $n'$-th row is the $n'$-th element of the $m'$-th column. We will denote concatenation of sequences by $+$.

Let $A$ be any array and $n$ be any natural number. For every $m$ smaller than the length of $A$'s columns and every $i$ smaller than the length of $A$, the $m$-parent of the $i$-th column is the last column before it whose $m$-th element is smaller than the $m$-th element of the $i$-th column, and which is an $(m-1)$-ancestor of the $i$-th column if $m>0$, if such a column exists. If no such column exists, then the $i$-th column does not have an $m$-parent. The $m$-ancestors (also called strict $m$-ancestors) of a column are its $m$-parent and the $m$-ancestors of its parent. The non-strict $m$-ancestors of a column are the column itself and its $m$-ancestors.

If $A$ is empty, then the expansion of $A$ at $n$ is $A[n]=A$. Otherwise let $C$ be the last element of $A$ and let $m_0$ be maximal such that $C$ has an $m_0$-parent, if such an $m_0$ exists, otherwise $m_0$ is undefined. Let arrays $G,B_0,B_1,...,B_n$ be such that:
\vspace{-8pt}
\begin{itemize}
    \setlength\itemsep{-4pt}
    \item $A=G+B_0+(C)$.
    \item The first element of $B_0$ is the $m_0$-parent of $C$ if $m_0$ is defined and otherwise $B_0$ is empty.
    \item For each $D$ in $B_0$ and $m<m_0$, if the first column in $B_0$ is $D$ or an $m$-ancestor of $D$, then it the $m$-th element of $D$ is said to ascend.
    \item $B_i$ is a copy of $B_0$, but for each ascending element of each column in $B_0$, its copy in $B_i$ is increased by $i\cdot((m$-th element of $C)-(m$-th element of the first column in $B_0))$, where $m$ is the index of the row in which that element is.
\end{itemize}
\vspace{-8pt}
Then the expansion $A[n]$ of $A$ at $n$ is $G+B_0+B_1+...+B_n$, with all rows of zeroes at the bottom removed.

$BMS$ is the closure of $\{((\underbrace{0,0,...,0,0}_n),(\underbrace{1,1,...,1,1}_n)) : n\in\mathbb{N}\}$ under expansion at each natural number, ordered by the $\subseteq$-minimal partial order such that $A[n]\le A$ for each $n\in\mathbb{N}$ and $A\in BMS$. Here, a partial order $\le$ is the set of pairs $(x,y)$ such that $x\le y$.
\end{definition}
\vspace{3pt}

This is the fourth official version of the system, which is why it is also referred to as $BM4$. The previous versions $BM1$, $BM2$ and $BM3$ were not well-founded, but as we prove below, $BM4$ is well-founded. There are also unofficial versions, of which $BM2.3$ is strongly believed to be equivalent to $BM4$ \cite{koteitan}, and $BM3.3$ is also notable for its similarity to $BM4$ and temporarily more predictable behavior. However, they are not the focus of this paper, so from now on, we will only refer to $BM4$.

The question of whether $BMS$ is well-ordered has been an open problem for almost 8 years, and it was among the most significant open problems in googology. Although $BMS$ is yet to be used outside of this field, its simplicity and large order type provide hope for future uses in proof theory and model theory. Before this paper, the research about $BMS$ has brought the following results:
\vspace{-8pt}
\begin{itemize}
    \setlength\itemsep{-4pt}
    \item $BMS$ restricted to arrays with one row is also called the Primitive Sequence system (or $PrSS$), and has a simple isomorphism with the iterated base-$\omega$ Cantor normal form - intuitively, each column represents a single $\omega$ in the string, the element of the column is the "height" of the $\omega$ (the number of exponents it appears in), and distinct $\omega$s with the same height are separated by a $+$ at the same level, unless there is an $\omega$ between them with a lower height. $\omega$s that do not have any $\omega$ in their exponent in the resulting string are exponentiated to $0$. This isomorphism can be proven easily by transfinite induction on the Cantor normal form expression, thus the order type of $PrSS$ is $\varepsilon_0$.
    \item $BMS$ restricted to arrays with two rows is also called the Pair Sequence System (or $PSS$), and was proven well-founded in 2018,\cite{pbot} with its order type shown to be the proof-theoretic ordinal of $\Pi^1_1-CA_0$, i.e. the countable collapse of $\omega_\omega$ using standard collapsing functions (such as Buchholz's function in this case).
\end{itemize}
\vspace{-8pt}

If we abbreviate $\langle L_\alpha,\in\rangle\prec_{\Sigma_1}\langle L_\beta,\in\rangle$ as $\alpha<_0\beta$, then informal estimates say that the order type of the set of arrays in $BMS$ smaller than $((0,0,0),(1,1,1),\linebreak[0](2,2,2))$ is most likely the supremum of, for each $n$, a recursive collapse (using standard collapsing functions) of the smallest ordinal $\alpha_0$ for which there exist $\alpha_1,\alpha_2,...,\alpha_n$ such that $\alpha_0<_0\alpha_1<_0\alpha_2<_0...<_0\alpha_n$.

The order type of the entirety of $BMS$ has not been carefully estimated in terms of ordinal functions yet, but is expected to be the supremum of, for each $n$, a recursive collapse of the smallest ordinal $\alpha$ for which there exists $\beta$ with $\langle L_\alpha,\in\rangle\prec_{\Sigma_n}\langle L_\beta,\in\rangle$, using collapsing functions that may be standard in the future.

Subjectively, $BMS$ is a very elegant way to represent large recursive ordinals. With enough formalization effort, it could give rise to a system of recursively large ordinals. This system would be similar to stability in structure and, as far as we know, similar to stability in scale too, but perhaps easier to understand or easier to use for some purposes such as ordinal analysis.

We utilize this similarity to prove that $BMS$ is well-ordered. Specifically, we first prove that $BMS$ is totally ordered and the order is precisely the lexicographical order. We then prove that a certain reflection property holds for stable ordinals. We show that this property allows us to map elements of $BMS$ to ordinals while preserving the order. Using this order-preserving function from $BMS$ to $Ord$, any infinite descending sequence in $BMS$ would be mapped to an infinite descending sequence in $Ord$, which cannot exist by definition, thus $BMS$ is well-ordered.

\section{The Proof}
Given that a property holds for every element of a set $X$, and that if it holds for $x$ then it holds for $f(x)$ for each $f$ in some set $F$ of functions, it is easy to see from the definition of closure that the property holds for all elements of the closure of $X$ under the functions $f\in F$. We consider this fact trivial enough to be used implicitly.

It is clear that for $A,A'\in BMS$, $A'<A$ iff $A$ is non-empty and $A'=A[n_0][n_1]...[n_m]$ for some $m,n_0,n_1,...,n_m\in\mathbb{N}$.

\begin{lemma}
For all $A\in BMS$ and $n\in\mathbb{N}$, $A[n]$ is lexicographically smaller than $A$ (with the columns also compared lexicographically).
\end{lemma}

\begin{proof}
Using the variable names from the definition of $BMS$, we have $A=G+B_0+(C)$ and $A[n]=G+B_0+B_1+...+B_n$. Then $A[n]<_{lex}A$ iff $B_1+B_2+...+B_n<_{lex}(C)$, which is trivial if $m_0$ is undefined (the empty sequence is lexicographically smaller than all other sequences, including $(C)$), and otherwise equivalent to the first column in $B_1$ being lexicographically smaller than $C$.

Let $R_i$ be the first column in $B_i$. Since $R_0$ is the $m_0$-parent of $C$, it is an $m$-ancestor of $C$ for each $m\le m_0$, thus the $m$-th element of $R_0$ is less than the $m$-th element of $C$. By definition, $R_1$ is a copy of $R_0$, but for each $m<m_0$, the $m$-th element is increased either by 0 or by the difference between itself and the $m$-th element of $C$. Then it is less than or equal to the $m$-th element of $C$, so the sequence of the first $m_0$ elements of $R_1$ is pointwise smaller than or equal to the sequence of the first $m_0$ elements of $C$ (in fact, it is equal, but that is not necessary for this proof). However, the $m_0$-th element of $R_1$ is necessarily equal to the $m_0$-th element of $R_0$ since $m_0<m_0$ is false, thus it is strictly smaller than the $m_0$-th element of $C$.

Therefore $R_1<_{lex}C$, which implies $B_1+B_2+...+B_n<_{lex}(C)$, and thus $A[n]<_{lex}A$.
\end{proof}

\begin{corollary}
For all $A,A'\in BMS$, $A'<A$ implies $A'<_{lex}A$.
\end{corollary}

\begin{lemma}
\label{totality}
$BMS$ is totally ordered.
\end{lemma}

\begin{proof}
For every non-empty $A\in BMS$, $A[0]$ is simply $A$ without the last column, as it is equal to $G+B_0$, and thus $A=A[0]+(C)$. Then it is trivial to prove by induction that for all $A,A'\in BMS$, if $A'$ is a subsequence of $A$, then $A'=A\underbrace{[0][0]...[0][0]}_n$ for some $n\in\mathbb{N}$, and thus $A'\le A$. Together with $A[n]$ being a subsequence of $A[n+1]$ for all $A\in BMS$ and $n\in\mathbb{N}$, this also means that for all $A,A'\in BMS$ and $n\in\mathbb{N}$, $A[n]\le A[n+1]$, and if $A[n]<A'\le A[n+1]$, then $A[n]\le A'[0]$. This implies that if some subset $X$ of $BMS$ is totally ordered, then $X\cup\{A[n] : A\in X\land n\in\mathbb{N}\}$ is also totally ordered. By induction, it is clear that if $X\subseteq BMS$ is totally ordered, then $X\cup\{A[n_0] : A\in X\land n_0\in\mathbb{N}\}\cup\{A[n_0][n_1] : A\in X\land n_0,n_1\in\mathbb{N}\}\cup...\cup\{A[n_0][n_1]...[n_m] : A\in X\land n_0,n_1,...,n_m\in\mathbb{N}\}$ is totally ordered for each $m\in\mathbb{N}$. Let $X_0=\{((\underbrace{0,0,...,0,0}_n),(\underbrace{1,1,...,1,1}_n)) : n\in\mathbb{N}\}$. Since each $A\in BMS$ is in $\{A''[n_0][n_1]...[n_m] : A''\in X_0\land n_0,n_1,...,n_m\in\mathbb{N}\}$ for some $m\in\mathbb{N}$, it is obvious that if $X_0$ is totally ordered, then for all $A,A'\in BMS$, there's some $m\in\mathbb{N}$ such that $A,A'\in X_0\cup\{A''[n_0] : A''\in X_0\land n_0\in\mathbb{N}\}\cup\{A''[n_0][n_1] : A''\in X_0\land n_0,n_1\in\mathbb{N}\}\cup...\cup\{A''[n_0][n_1]...[n_m] : A''\in X_0\land n_0,n_1,...,n_m\in\mathbb{N}\}$, which is totally ordered, and therefore $A,A'$ are comparable. So if $X_0$ is totally ordered, then $BMS$ is totally ordered.

It is now sufficient to prove that $X_0$ is totally ordered. This is easy, since $((\underbrace{0,0,...,0,0}_{n+1}),(\underbrace{1,1,...,1,1}_{n+1}))[1]$ is trivially $((\underbrace{0,0,...,0,0}_n),(\underbrace{1,1,...,1,1}_n))$, and thus by induction, for each $n<m\in\mathbb{N}$, $((\underbrace{0,0,...,0,0}_n),(\underbrace{1,1,...,1,1}_n))=((\underbrace{0,0,...,0,0}_m),\linebreak(\underbrace{1,1,...,1,1}_m))\underbrace{[1][1]...[1][1]}_{m-n}<((\underbrace{0,0,...,0,0}_m),(\underbrace{1,1,...,1,1}_m))$, and all elements of $X_0$ are of this form, so all elements of $X_0$ are pairwise comparable.
\end{proof}

\begin{corollary}
The ordering of $BMS$ coincides with the lexicographical ordering with columns compared lexicographically.
\end{corollary}

\begin{lemma}
\label{bpiso}
Let $A$ be a non-empty array and $n$ be a natural number, let $G,B_0,B_1,...,B_n,m_0$ be as in Definition \ref{bmsdef}, and let $l_0,l_1$ be the lengths of $G,B_0$.\newline(i) For all $i<l_0$, $j<l_1$ and $k\in\mathbb{N}$, in $A[n]$, the $i$-th column in $G$ is a $k$-ancestor of the $j$-th column in $B_0$ iff it is a $k$-ancestor of the $j$-th column in $B_n$.\newline(ii) For all $i,j<l_1$ and $k\in\mathbb{N}$, the $i$-th column in $B_0$ is a $k$-ancestor of the $j$-th column in $B_0$ iff the $i$-th column in $B_n$ is a $k$-ancestor of the $j$-th column in $B_n$.\newline(iii) If $n>0$, then for all $i<l_1$ and $k<m_0$, in $A$, the $i$-th column in $B_0$ is a $k$-ancestor of the last column of $A$ iff in $A[n]$, the $i$-th column in $B_{n-1}$ is a $k$-ancestor of the first column in $B_n$.\newline(iv) For all $0<i<l_1$ and $k\in\mathbb{N}$, in $A[n]$, the $k$-parent of the $i$-th column in $B_n$ is either in $B_n$ or in $G$.\newline(v) For all $i,j<l_1$ and $k\in\mathbb{N}$ and $n_0<n_1<n$, in $A[n]$, the $i$-th column in $B_{n_0}$ is a $k$-ancestor of the $j$-th column in $B_{n_1}$ iff it's a $k$-ancestor of the $j$-th column in $B_{n_1+1}$.
\end{lemma}

\begin{proof}
We can prove this by induction on $k$. The proof is relatively straightforward, but tedious. The author recommends drawing the mentioned ancestry relations in order to see what is happening.

Assume all 5 statements hold for all $k'<k$.

For $(ii)$, fix $i$ and $j$. If $j=0$ then it is trivial, so we will only consider the case $j>0$. From the assumption, it follows that for all $k'<k$ and $i'<l_0$, the $i'$-th column in $B_0$ is a $k'$-ancestor of the $j$-th column in $B_0$ iff the $i'$-th column in $B_n$ is a $k'$-ancestor of the $j$-th column in $B_n$. Let $I$ be the set of $i'$ such that for all $k'<k$, the $i'$-th column in $B_0$ is a $k'$-ancestor of the $j$-th column in $B_0$.

Since for all $k'<k$, $k'$-ancestry is a total order on the columns with indices in $I$, the $k$-parent of each such column is simply the last such column before it with a smaller $k$-th element. The $k$-th element of the $j$-th column in $B_0$ ascends iff the first column in $B_0$ is in $I$ and is a $k$-ancestor of the $j$-th column in $B_0$, which is equivalent to the $k$-th element of the first column in $B_0$ being smaller than the $k$-th element of all columns between it and the $j$-th column in $B_0$, so it is also a $k$-ancestor of all other columns with indices in $I$. This means that either the $k$-th elements of all columns in $B_0$ with indices in $I$ ascend or the $k$-th element of the $j$-th column in $B_0$ doesn't ascend.

In the first case, the differences between the columns in $B_n$ with indices in $I$ are the same as in $B_0$, and since $k'$-ancestry relations between them are also the same as in $B_0$ for $k'<k$, $k$-ancestry must be the same too, because everything it depends on is the same. In the second case, since the $j$-th column doesn't ascend, in $B_n$, there trivially cannot be any $k$-ancestors of the $j$-th column that aren't copies of $k$-ancestors of the $j$-th column in $B_0$. Since this possibility requires that the first column in $B_0$ is not a $k$-ancestor of the $j$-th column, it is also not a $k$-ancestor of any $k$-ancestor of the $j$-th column, thus the $k$-th elements of the $k$-ancestors of the $j$-th column also don't ascend, and therefore the differences between them are the same, implying that the $k$-ancestry relations are preserved. Either way, $(ii)$ holds for $k$.

The above can trivially be extended to include the next copy of the first column in $B_0$, and then since the $k$-th element of the first column in $B_1$ is easily seen to be the same as $C$ as long as $k<m_0$, $(iii)$ holds for $k$.

Then to prove $(iv)$, we first observe that if for some $k'<k$, the $k'$-parent of the $i$-th column in $B_n$ is in $G$, then all of its $k'$-ancestors are in $G$, and its $k$-parent must be one of its $k'$-ancestors so it is also in $G$. So we're left with the case that for all $k'<k$, the $k'$-parent of the $i$-th column in $B_n$ is in $B_n$.

If the first column of $B_n$ is a $k'$-ancestor of the $i$-th column in $B_n$ for all $k'<k$, and yet its $k$-parent is not in $B_n$ or $G$, then the first column in $B_n$ is not a $k$-ancestor of the $i$-th column in $B_n$. Therefore from $(ii)$ for $k$, which we have already proven, we get that the $k$-parent of the $i$-th column in $B_0$ is not in $B_0$ (therefore it is in $G$), which also implies that the $k$-th element of the $i$-th column in $B_0$ does not ascend in the expansion of $A$, so it is equal to the $k$-th element of the $i$-th column in $B_n$. But from $(ii)$ for all $k'<k$ and the fact that the first column in $B_n$ is a $k'$-ancestor of the $i$-th column in $B_n$ for all $k'<k$, we get that the first column in $B_0$ is a $k'$-ancestor of the $i$-th column in $B_0$ for all $k'<k$.

This, together with its $k$-parent being in $G$, means that for all columns in $B_0$ that are $k'$-ancestors of the $i$-th column in $B_0$ for all $k'<k$, their $k$-th element is at least as large as the $k$-th element of the $i$-th column in $B_0$, and therefore at least as large as the $k$-th element of the $i$-th column in $B_n$. This includes the first column in $B_0$, and since the $k$-th element of the first column in $B_n$ is by definition at least as large as the $k$-th element of the first column in $B_0$, which is at least as large as the $k$-th element of the $i$-th column in $B_n$, which is by definition strictly larger than the $k$-th element of the $k$-parent of the $i$-th column in $B_n$, we get that the $k$-th element of the first column in $B_n$ is strictly larger than the $k$-th element of the $k$-parent of the $i$-th column in $B_n$. With that, and due to the facts that $k'$-ancestry is a total order on the set of $k'$-ancestors of each column for each $k'$, and that both the first column in $B_n$ and the $k$-parent of the $i$-th column in $B_n$ are $k'$-ancestors of the $i$-th column in $B_n$ for every $k'<k$, and the latter is before the former, we get that the $k$-parent of the $i$-th column in $B_n$ is also a $k$-ancestor of the first column in $B_n$.

If $k\ge m_0$ (using variable names from Definition \ref{bmsdef}), then this is already a contradiction, because the $m_0$-parent of the first column in $B_n$ is easily seen to be in $G$. Otherwise, let $n'<n$ be the natural number such that the $k$-parent of the $i$-th column in $B_n$ is in $B_{n'}$. From repeated applications of $(iii)$ for $k$, which we have already proven, we get that the first column in $B_{n'}$ is a $k$-ancestor of the first column in $B_n$, and therefore by $k$-ancestry being a total order on the set of $k$-ancestors of the first column in $B_n$, we get that the first column in $B_{n'}$ is a $k$-ancestor of the $k$-parent of the $i$-th column in $B_n$, and thus is also a $k$-ancestor of the $i$-th column in $B_n$. This, however, by more repeated applications of $(iii)$, implies that the first column in $B_0$ is a $k$-ancestor of the $i$-th column in $B_n$, which is in contradiction with the fact that the $k$-th element of the first column in $B_0$ is at least as large as the $k$-th element of the $i$-th column in $B_n$.

Now, for $(iv)$, we're left with the case that for some $k'<k$, the first column in $B_n$ is not a $k'$-ancestor of the $i$-th column in $B_n$. However, if we choose a specific such $k'<k$, then by $(iv)$ for $k'$ we get that the $k'$-parent of every $k'$-ancestor in $B_n$ of the $i$-th column in $B_n$ is either in $B_n$ or in $G$, from which it follows that all $k'$-ancestors of the $i$-th column in $B_n$ are either in $B_n$ or in $G$, and that includes the $k$-parent of the $i$-th column in $B_n$, proving $(iv)$ for $k$.

With $(iv)$ proven for $k$, the proof of $(ii)$ and $(iii)$ for $k$ can also be easily modified for relations between $G$ and $B_0$ and between $G$ and $B_n$, with all nontrivialities accounted for by $(iv)$ for $k$: either the $k$-th element of the $j$-th column in $B_0$ ascends and the $j$-th column in $B_n$ trivially has the first column in $B_0$ as a $k$-ancestor, thus the $k$-ancestors in $G$ are simply the $k$-ancestors of that (by totality of $k$-ancestry on the set of $k$-ancestors of the $j$-th column in $B_n$), or the $k$-th element of the $j$-th column in $B_0$ doesn't ascend and $B_n$'s copy $C_n$ of the ($j$-th column in $B_0$)'s first non-strict $k$-ancestor $C_0$ in $B_0$ is easily seen to have the same $k$-parent as $C_0$, because the $k$-parents of $C_0$ and $C_n$ are both in $G$, the $k$-th elements of $C_0$ and $C_n$ are equal, and the sets of $k'$-ancestors of $C_0$ and of $C_n$ are the same for every $k'<k$ by $(i)$ for $k'$, thus the $k$-ancestors in $G$ of both $C_0$ and $C_n$ are that $k$-parent and its $k$-ancestors. Therefore $(i)$ also holds for $k$.

Finally, $(v)$ can be proven for $k$ by simply letting $\{n_2,n_3\}=\{n_1,n_1+1\}$ (the two options together give the proofs of both directions of $(v)$), and noticing that if the $j$-th column in $B_{n_2}$ has a $k$-ancestor in $B_{n_0}$, then the first column in $B_{n_2}$ must also be its $k$-ancestor (similarly to the reasoning near the end of the previous paragraph - using $(iv)$ for the last $k$-ancestor in $B_{n_2}$ of the $j$-th column in $B_{n_2}$), and therefore by totality of $k$-ancestry on the set of $k$-ancestors of the $j$-th column in $B_{n_2}$, the $i$-th column in $B_{n_0}$ is a $k$-ancestor of the first column in $B_{n_2}$. Then if $k\ge m_0$, we get a contradiction, because the $k$-ancestors of the $j$-th column in $B_{n_2}$ are all in $B_{n_2}$ or $G$, as we've already proven, so it must be that $k<m_0$. In that case, by application of $(iii)$ and either (depending on $n_2-n_1$) another application of the totality of $k$-ancestry on the set of $k$-ancestors of the $j$-th column in $B_{n_2}$ or an application of transitivity of $k$-ancestry, we get that the $i$-th column in $B_{n_0}$ is also a $k$-ancestor of the first column in $B_{n_3}$, and finally by an application of $(ii)$ for $k$, we get that the first column in $B_{n_3}$ is a $k$-ancestor of the $j$-th column in $B_{n_3}$, so by transitivity of $k$-ancestry, the $i$-th column in $B_{n_0}$ is a $k$-ancestor of the $j$-th column in $B_{n_3}$, which concludes the proof of $(v)$ for $k$.

By induction, all 5 statements in the lemma always hold.
\end{proof}

We will abbreviate $\langle L_\alpha,\in\rangle\preceq_{\Sigma_{n+1}}\langle L_\beta,\in\rangle$ as $\alpha\le_n\beta$, and similarly for the strict versions of these relations. Here, $L_\alpha$ is the $\alpha$-th level of the constructible hierarchy, and $M\preceq_{\Sigma_n}N$ means that $M$ is a $\Sigma_n$-elementary substructure of $N$.

Let $\sigma$ be the smallest ordinal $\alpha$ such that there exists an ordinal $\beta$ with $\forall n\in\mathbb{N}(\alpha<_n\beta)$.

\begin{lemma}
\label{stbrflprop}
For all $\alpha,\beta\in\sigma$ and $n\in\mathbb{N}$, if $\omega<\alpha<_n\beta$, then for all finite $X,Y\subseteq Ord$ such that $\gamma<\alpha\le\delta<\beta$ for all $\gamma\in X$ and $\delta\in Y$, there exists a finite $Y'\subseteq Ord$ and a bijection $f: Y\to Y'$ such that for all $\gamma\in X$, all $\delta_0,\delta_1\in Y$, all $k\in\mathbb{N}$ and all $m<n$:
\vspace{-8pt}
\begin{itemize}
    \setlength\itemsep{-4pt}
    \item $\gamma<f(\delta_0)<\alpha$
    \item $\gamma<_k\delta_0\Rightarrow \gamma<_kf(\delta_0)$
    \item $\delta_0<\delta_1\Rightarrow f(\delta_0)<f(\delta_1)$
    \item $\delta_0<_k\delta_1\Rightarrow f(\delta_0)<_kf(\delta_1)$
    \item $\delta_0<_m\beta\Rightarrow f(\delta_0)<_m\alpha$
\end{itemize}
\vspace{-8pt}
\end{lemma}

We can prove this by constructing a $\Sigma_{n+1}$ formula that, when interpreted in $L_\beta$, asserts all the true instances of the statements on the left side of the implications, and when interpreted in $L_\alpha$, asserts the corresponding instances of the statements on the right side of the implications. One small issue is the first assertion, which is unconditional. However, the $f(\delta_0)<\alpha$ part is simply asserting that $f(\delta_0)$ exists in $L_\alpha$, which will be done by existentially quantifying the variable, and since $\gamma<\alpha\le\delta_0$ is necessarily true, $\gamma<f(\delta_0)$ is equivalent to $\gamma<\delta_0\Rightarrow\gamma<f(\delta_0)$, which is a conditional statement.

\begin{proof}
We construct a formula with parameters $\gamma_0,\gamma_1,...,\gamma_{|X|-1}$, which are all the elements of $X$. Since they're ordinals smaller than $\alpha$, they are in $L_\alpha$, therefore we can use them as parameters in a formula that we want to reflect using the stability relation between $\alpha$ and $\beta$.

Let $\varphi_0(\eta,\xi)$ be a formula asserting $\eta<\xi$. Let $\varphi_1(\eta,\xi,k)$ be a formula asserting $\eta<_k\xi$. Let $\varphi_2(\eta,k)$ be a formula asserting $\eta<_kOrd$, i.e. $\langle L_\eta,\in\rangle\prec_{\Sigma_{k+1}}\langle L,\in\rangle$.

$\varphi_0$ is clearly $\Sigma_0$, as it is simply the atomic formula $\eta\in\xi$. This means it is $\Sigma_{n+1}$. $\varphi_1$ only needs to assert the existence of $L_\xi$, the defining characteristics of it (specifically that it is a level of $L$, which is simply $V=L$ relativized to it, and that the ordinals in it are precisely the elements of $\xi$, which is trivially $\Sigma_0$), and then it needs to assert that $\varphi_2(\eta,k)$ relativized to $L_\xi$ holds. The relativization of a first-order formula to a set is trivially always $\Sigma_0$. Assuming $\varphi_2$ is first-order, the only unbounded quantifier in $\varphi_1$ is the one existentially quantifying $L_\xi$. Then $\varphi_1$ is $\Sigma_1$, which means it's also $\Sigma_{n+1}$. Finally, $\varphi_2(\eta,k)$ is $\Pi_{k+1}$, as shown in \cite{kranakis} (Theorem 1.8), which means it is $\Sigma_{k+2}$, and therefore first-order. In all non-relativized uses of $\varphi_2$, we will require $k<n$, which means $k+2\le n+1$, thus it is $\Sigma_{n+1}$.

$X$ and $Y$ are finite, and all of their elements are smaller than $\sigma$ so for each $\eta,\xi\in X\cup Y$, there are only finitely many $k$ for which $\varphi_1(\eta,\xi,k)$ is true. Then there are finitely many instances of $\varphi_0(\gamma_i,\delta_j)$, $\varphi_1(\gamma_i,\delta_j,k)$, $\varphi_0(\delta_i,\delta_j)$, $\varphi_1(\delta_i,\delta_j,k)$ and $\varphi_2(\delta_i,m)$ with $k\in\mathbb{N}$ and $m<n$, which are true when each $\delta_i$ is interpreted as the $i$-th element of $Y$. So their conjunction $\varphi$ is a conjunction of finitely many $\Sigma_{n+1}$ formulae, therefore it is itself a $\Sigma_{n+1}$ formula. Then we only need a $\Sigma_{n+1}$ formula $\psi$ asserting that all the $\delta_i$ are ordinals, which is trivial.

Now, the formula $\psi\land\varphi$ is $\Sigma_{n+1}$, therefore the formula $\exists\delta_0,\delta_1,...,\delta_{|Y|-1}(\psi\land\varphi)$ is also $\Sigma_{n+1}$. In $L_\beta$, the witnesses of that existential quantifier are the elements of $Y$, therefore the formula is true in $L_\beta$. Then by $\alpha<_n\beta$, it must be true in $L_\alpha$, and since it encodes all the relations between elements of $X$, elements of $Y$ and $\beta$ that need to be reflected to relations between elements of $X$, elements of $Y'$ and $\alpha$, the witnesses of that formula in $L_\alpha$ form a set $Y'$ that, together with the unique order isomorphism $f: Y\to Y'$, satisfies the conditions in the lemma.
\end{proof}

Note that this reflection is similar to reflection in Patterns of Resemblance, and those could be used too. However, the author is not as experienced in working with Patterns of Resemblance, so it was easier to use stability.

\begin{theorem}
$BMS$ is well-ordered.
\end{theorem}

\begin{proof}
We will define a function $o: BMS\to Ord$ in the following way. Consider an array $A$ with length $n$. A stable representation of $A$ is a function $f: n\to Ord$ such that for all $i,j<n$, $i<j\Rightarrow f(i)<f(j)$ and for all $m$, if the $i$-th column of $A$ is an $m$-ancestor of the $j$-th column of $A$, then $f(i)<_mf(j)$. Let $o(A)$ be the minimal $\alpha\in Ord$ such that for some stable representation $f$ of $A$, all outputs of $f$ are smaller than $\alpha$.

This proof is similar to the proof of Lemma \ref{totality} - we prove by induction on the number of expansions needed to reach an array, that $o$ is defined and order-preserving on all of $BMS$, by starting from $X_0$ and proving that if it holds for some $Z$, then it holds for $Z\cup\{A[n] : A\in Z\land n\in\mathbb{N}\}$, and using the fact that every pair $A,A'$ of arrays is reached after finitely many applications of this induction step.

Of course, $o(A)$ is defined for $A\in X_0=\{((\underbrace{0,0,...,0,0}_n),(\underbrace{1,1,...,1,1}_n)) : n\in\mathbb{N}\}$, and it is easy to see that $o(((\underbrace{0,0,...,0,0}_n),(\underbrace{1,1,...,1,1}_n)))<o(((\underbrace{0,0,...,0,0}_m),\linebreak(\underbrace{1,1,...,1,1}_m)))$ iff $n<m$, thus $o$ is also order-preserving in this set.

Let $Z$ be a set of arrays, on which $o$ is order-preserving and defined for all of $Z$'s elements. Let $A\in Z$. If $A$ is empty, then trivially for all $n\in\mathbb{N}$, $o(A[n])=o(A)\le o(A)$, and thus $o(A[n])$ is also defined and order is preserved. Otherwise, let $f$ be a stable representation of $A$ whose outputs are all smaller than $o(A)$. We can then recursively define stable representations of $A[n]$ in the following way.

Let $l_n$ be the length of $A[n]$ for all $n\in\mathbb{N}$. A stable representation $f_0$ of $A[0]$ is simply $f$ restricted to $l_0$. Using variable names from the definition of $BMS$, this representation trivially maps indices (in $A[0]$) of columns in $B_0$ to the ordinals to which $f$ maps indices (in $A$) of columns in $B_0$.

Let $f_n$ be a stable representation of $A[n]$ that maps indices (in $A[n]$) of columns in $B_n$ to the ordinals to which $f$ maps indices (in $A$) of columns in $B_0$. Then using the reflection property from Lemma \ref{stbrflprop} with $\alpha$ being the ordinal to which $f_n$ maps the index of the first column in $B_n$, $\beta$ being the ordinal to which $f$ maps the last column of $A$, $X$ being the set of ordinals to which $f_n$ maps indices of columns before $B_n$, and $Y$ being the set of ordinals to which $f_n$ maps indices of columns in $B_n$ (or to which $f$ maps indices of columns in $B_0$), we get a set $Y'$ of ordinals due to $\alpha<_{m_0}\beta$. We can then define $f_{n+1}$ by making it the same as $f_n$ for indices of columns before $B_n$, mapping indices of columns in $B_n$ to the elements of $Y'$, and mapping indices (in $A[n+1]$) of columns in $B_{n+1}$ to the elements of $Y$.

It follows from Lemma \ref{bpiso} that $f_{n+1}$ is a stable representation of $A[n+1]$. Then it's trivially a stable representation of $A[n+1]$ that maps indices of columns in $B_{n+1}$ to the ordinals to which $f$ maps indices of columns in $B_0$, therefore by induction, for all $m\in\mathbb{N}$, there is a stable representation of $A[m]$ that maps indices of columns in $B_m$ to the ordinals to which $f$ maps indices of columns in $B_0$. Since all these ordinals are smaller than $\beta$, $o(A[m])$ is defined and is at most $\beta$, and since $\beta$ is an output of $f$, it is smaller than $o(A)$, so $o(A[m])<o(A)$, which means $o$ is defined and order-preserving (due to the order being originally defined only by comparing an array with its expansions) on $Z\cup\{A[m] : A\in Z\land m\in\mathbb{N}\}$.

Now, similarly to the proof of Lemma \ref{totality}, with $X_0=\{((\underbrace{0,0,...,0,0}_n),\linebreak(\underbrace{1,1,...,1,1}_n)) : n\in\mathbb{N}\}$, we conclude that $o$ is defined and order-preserving on $X_0\cup\{A[n_0] : A\in X_0\land n_0\in\mathbb{N}\}\cup\{A[n_0][n_1] : A\in X_0\land n_0,n_1\in\mathbb{N}\}\cup...\cup\{A[n_0][n_1]...[n_m] : A\in X_0\land n_0,n_1,...,n_m\in\mathbb{N}\}$ for each $m\in\mathbb{N}$, and since all $A,A'\in BMS$ are also in this set for some $m$, $o$ is defined for them their order is preserved by $o$, so $o$ is defined and order-preserving on all of $BMS$.

Then if $BMS$ was not well-ordered, there would be an infinite descending sequence in $BMS$, which would get mapped to an infinite descending sequence of ordinals by $o$, and that cannot exist by the definition of ordinals. Therefore $BMS$ is well-ordered.
\end{proof}

\section{Future research}
We hope to use $BMS$ in ordinal analysis, first using it to rewrite analyses of theories that have already been analyzed by other means, and then analyzing even stronger theories, ideally up to full second-order arithmetic if the order type of $BMS$ is large enough for that.

Once this approach proves viable, we also plan to continue proving the well-orderedness of similar notation systems with larger order types, such as $Y$ sequence \cite{naruyoko1} and its extension $\omega-Y$ sequence \cite{naruyoko2}.

Another challenge that is relevant is the task to find a "self-contained" proof of well-orderedness of $BMS$ (that is, a proof using only concepts that are directly related to $BMS$, which excludes stability and ordinal collapsing functions), as this would simplify the translation of the proof to theories that only deal with basic structures, such as third-order arithmetics.

\section{Acknowledgements}
I would like to express my deepest gratitude to the discord user C7X (also known as Convindix) for proofreading this paper, as well as introducing me to the concept of stability years ago, which led to this paper's very existence.

I also want to thank the googology and apeirology community for the doubt that motivated me to finish this paper.

\newpage
\section{References}
\printbibliography
\end{document}